\documentclass[12pt]{amsart}
\usepackage{amssymb}
\usepackage{mathtools}
\usepackage{xcolor}

\newcommand{\REM}[1]{\relax}

\newcommand{\Card}{\mathsf{Card}}

\makeatletter
\newcommand*{\Log}{\mathop{\operator@font Log}\nolimits}
\newcommand*{\Arg}{\mathop{\operator@font Arg}\nolimits}
\newcommand*{\tg}{\mathop{\operator@font tg}\nolimits}
\newcommand*{\ctg}{\mathop{\operator@font ctg}\nolimits}
\newcommand*{\cosec}{\mathop{\operator@font cosec}\nolimits}
\newcommand*{\arctg}{\mathop{\operator@font arctg}\nolimits}
\newcommand*{\arcctg}{\mathop{\operator@font arcctg}\nolimits}
\newcommand*{\sh}{\mathop{\operator@font sh}\nolimits}
\newcommand*{\ch}{\mathop{\operator@font ch}\nolimits}
\makeatother

\begin{document}
\numberwithin{equation}{section}

\newcommand{\anglim}{\angle\lim}

\newcommand{\de}{\partial}
\newcommand{\Hol}{{\sf Hol}}
\newcommand{\Aut}{{\sf Aut}(\mathbb D)}
\newcommand{\loc}{L^1_{\rm{loc}}}
\newcommand{\Moeb}{\mathsf{M\ddot ob}}

\newcommand{\Mobius}{M\"obius }

\def\Re{{\sf Re}\,}
\def\Im{{\sf Im}\,}

%Pavel's macros
\newcommand{\UH}{\mathbb{H}}
\newcommand{\UHi}{\mathbb{H}_i}
\newcommand{\Real}{\mathbb{R}}
\newcommand{\Natural}{\mathbb{N}}
\newcommand{\Complex}{\mathbb{C}}
\newcommand{\ComplexE}{\overline{\mathbb{C}}}
\newcommand{\Int}{\mathbb{Z}}
\newcommand{\UD}{\mathbb{D}}
\newcommand{\UC}{{\partial\UD}}
\newcommand{\clS}{\mathcal{S}}
\newcommand{\U}{{\mathfrak U}}
\newcommand{\Ut}[1]{{\mathfrak U\hskip.1em}_{#1}\hskip-.09em}
\newcommand{\UF}{{\mathfrak U\hskip.08em}}
\newcommand{\Us}{{\mathfrak U\hskip.03em}}
\newcommand{\Un}{{\mathfrak U\hskip.07em}_0\hskip-.07em}
\newcommand{\Uone}{{\mathfrak U\hskip.05em}_1\hskip-.07em}

\newcommand{\Maponto}%{\stackrel{%\scriptscriptstyle
%\!\!\mathsf{onto}\,\,}{-\!\!\!\longrightarrow\vphantom{\to}}}
{\xrightarrow{\scriptstyle \!\mathsf{onto}\,}}

\let\N=\Natural
\let\R=\Real
\let\D=\UD
\let\RS=\ComplexE
\let\C=\Complex
\newcommand{\uH}{\mathbb{H}_{\mathrm{i}}}
\newcommand{\rH}{\mathbb{H}_1}
\newcommand{\mscrm}{\mathfrak{m}}
\newcommand{\mscrc}{\mathfrak{c}}
\newcommand{\mathup}[1]{\mathrm{#1}}

\newcommand{\uHc}{\overline{\mathbb{H}}_{\mathrm{i}}}

\newcommand{\Step}[2]{\begin{itemize}\item[{\bf Step~#1.}]\textit{#2}\end{itemize}}
\newcommand{\step}[2]{\begin{itemize}\item[\textit{Step\,#1:}]\textit{#2}\end{itemize}}
\newcommand{\case}[2]{\begin{itemize}\item[\textit{Case~#1:}]\textit{#2}\end{itemize}}
\newcommand{\caseM}[1]{\par\vskip.7ex\noindent\textit{Case~#1}.~}
\newcommand{\caseU}[2]{\vskip1.4ex\begin{trivlist}%
\item[\textbf{CASE~#1}:]\textbf{{\mathversion{bold}#2}}\end{trivlist}\vskip-0.3ex}
\newcommand{\parM}[1]{\par\vskip.7ex\noindent\textit{#1}.~}
\newcommand{\proofbox}{\hfill$\Box$}
\newcommand{\trivlistcorr}{\leftmargin=2.3em\labelsep=0em\labelwidth=2.3em\rightmargin=.5em\itemindent=0em}
\newcommand{\trivlistcorrM}{\leftmargin=2.3em\labelsep=0.5em\labelwidth=1.8em\rightmargin=.5em\itemindent=0em}
\newcommand{\claim}[2]{\vskip1.4ex\begin{list}{}{\trivlistcorr}\item[#1\hfill] {\it #2}\end{list}\vskip.5ex}
\newcommand{\claimM}[2]{\vskip1.4ex\begin{list}{}{\trivlistcorrM}\item[\textbf{Claim #1}:~] {\it #2}\end{list}}

\newcommand{\STOP}{\par\hbox to\textwidth{\color{red}\leaders\hbox{\,STOP\,}\hfil}\par}

\newcommand{\mcite}[1]{\csname b@#1\endcsname}

\newcommand{\di}{\mathrm{d}}
\newcommand{\DI}{\,\di}
\newcommand{\ddt}[1]{\frac{\di #1}{\di t}}
\newcommand{\DDT}{\ddt{}\,}

\newcommand{\diam}{\mathop{\mathsf{diam}}\nolimits}

\theoremstyle{theorem}
\newtheorem {result} {Theorem}
\setcounter {result} {64}
 \renewcommand{\theresult}{\char\arabic{result}}

%\newcommand{\Step}[2]{\begin{itemize}\item[{\bf Step~#1.}]{\it #2}\end{itemize}}

%\newcommand{\proofbox}{\hfill$\Box$}
%end of Pavel's macros block

%New macros Pavel
\newcommand{\Spec}{\Lambda^d}
\newcommand{\SpecR}{\Lambda^d_R}
\newcommand{\Prend}{\mathrm P}
%end of New macros Pavel

%\tableofcontents

%\def\Label#1{\label{#1}{\bf (#1)}~}
%\def\Label#1{\label{#1}}

% Standard sets

\def\cn{{\C^n}}
\def\cnn{{\C^{n'}}}
\def\ocn{\2{\C^n}}
\def\ocnn{\2{\C^{n'}}}
\def\je{{\6J}}
\def\jep{{\6J}_{p,p'}}
\def\th{\tilde{h}}

% Abbreviations

\def\dist{{\rm dist}}
\def\const{{\rm const}}
\def\rk{{\rm rank\,}}
\def\id{{\sf id}}
\def\aut{{\sf aut}}
\def\Aut{{\sf Aut}}
\def\CR{{\rm CR}}
\def\GL{{\sf GL}}
\def\Re{{\sf Re}\,}
\def\Im{{\sf Im}\,}

\def\Cap{\mathop{\mathrm{cap}}}

\def\la{\langle}
\def\ra{\rangle}

\newcommand{\sgn}{\mathop{\mathrm{sgn}}}

\emergencystretch15pt \frenchspacing

\newtheorem{theorem}{Theorem}
\newtheorem{lemma}{Lemma}[section]
\newtheorem{proposition}{Proposition}
\newtheorem{corollary}{Corollary}

\theoremstyle{definition}
\newtheorem{definition}[lemma]{Definition}
\newtheorem{example}[lemma]{Example}

\theoremstyle{remark}
\newtheorem{remark}[lemma]{Remark}
\numberwithin{equation}{section}

\newcommand{\res}{\mathop{\mathrm{\,r\,e\,s\,}}\limits}

\newcommand{\distCE}{\mathop{\mathsc{dist}_\ComplexE}}

\newenvironment{mylist}{\begin{list}{}%
{\labelwidth=2em\leftmargin=\labelwidth\itemsep=.4ex plus.1ex
minus.1ex\topsep=.7ex plus.3ex
minus.2ex}%
\let\itm=\item\def\item[##1]{\itm[{\rm ##1}]}}{\end{list}}

\title[Parametric representations and boundary fixed points]{Parametric representations and boundary fixed points of univalent self-maps of the unit disk}
\author[P. Gumenyuk]{Pavel Gumenyuk\,$^*$}
\address{Department of Mathematics and Natural Sciences, University of Stavanger, N-4036 Stavanger, Norway.}
\email{pavel.gumenyuk@uis.no}



\thanks{$^*$Partially supported by {\it Ministerio de Econom\'\i{}a y Competitividad} (Spain) project
\hbox{MTM2015-63699-P}.}

\begin{abstract}
A classical result in the theory of Loewner's parametric representation states that
the semigroup~$\Us_*$ of all conformal self-maps $\phi$ of the unit disk $\D$ normalized by $\phi(0) = 0$ and
$\phi'(0) > 0$ can be obtained as the reachable set of the Loewner\,--\,Kufarev control system
$$
\frac{\di w_t}{\di t}=G_t\circ w_t,\quad t\geqslant0,\qquad w_0=\id_\UD,
$$
where the control functions $t\mapsto G_t\in\Hol(\UD,\C)$ form a certain convex cone.
Here we extend this result to the semigroup $\UF[F]$ consisting of all conformal~$\phi:\UD\to\UD$ whose set of boundary regular fixed points contains a given finite set $F\subset\UC$ and to its subsemigroup $\Ut\tau[F]$ formed by $\id_\UD$ and all~$\phi\in\UF[F]\setminus\{\id_\UD\}$ with the prescribed boundary Denjoy\,--\,Wolff point~$\tau\in\UC\setminus F$. This completes the study launched in~\cite{Gumenyuk_parametric}, where the case of interior Denjoy\,--\,Wolff point~$\tau\in\UD$ was considered.
\end{abstract}

\maketitle

\tableofcontents

\let\ge=\geqslant
\let\le=\leqslant

\section{Introduction}
\noindent One of the cornerstone results of Loewner's parametric representation can be stated in the following form, see, e.g. \cite[Problem~3 on p.\,164]{Pommerenke} or \cite[pp.\,69-70]{Aleksandrov}.
\begin{result}\label{TH_classical}
A function $\phi$ defined in the unit disk~$\UD:=\{z\colon|z|<1\}$ is a univalent (i.e. injective and holomorphic) self-map of~$\UD$ normalized by $\phi(0) = 0$, $\phi'(0) > 0$, if and only if it can be represented in the form $\phi(z)=w_z(T)$ for all~$z\in\UD$, where $T\ge0$ and $w=w_z(t)$ is the unique solution to initial value problem
\begin{equation}\label{EQ_classical-LK-ODE}
\frac{dw(t)}{dt}=-w(t)\,p\big(w(t),t\big),\quad t\ge0;\qquad w(0)=z,
\end{equation}
with a function $p:\UD\times[0,+\infty)\to\C$ holomorphic in the first argument, measurable in the second argument and satisfying $\Re p(\cdot,t)>0$ and $p(1,t)=1$ for a.e.~$t\ge0$.
\end{result}
A related result, see, e.g. \cite[\S6.1]{Pommerenke}, stating that \textit{the class $\clS$ of all univalent mappings $f:\UD\to\C$ normalized by~$f(0)=f'(1)-1=0$ coincides with the set of all functions representable as $$f(z)=\lim_{t\to+\infty}e^t w_z(t),$$ where $w_z$, as above, is the solution to~\eqref{EQ_classical-LK-ODE},} is much better known and usually meant when one talks about \textsl{Loewner's parametric representation} of univalent functions. However, it is the former version of Loewner's parametric method, i.e. Theorem~\ref{TH_classical}, which will play more important role in the frames of the present study.

There is no natural linear structure compatible with the property of injectivity. In fact, the class $\clS$ mentioned above even fails to be a convex subset of~$\Hol(\UD,\C)$. A similar statement holds for the class $\Us_*$ of all univalent self-maps $\phi:\UD\to\UD$, $\phi(0)=0$, $\phi'(0)>0$, involved in Theorem~\ref{TH_classical}.

That is why in many problems for univalent functions, standard variation and optimization methods cannot be applied directly. One has to use a suitable parametric representation to replace univalent mappings by a ``parameter'', such as the driving function~$p$ in~\eqref{EQ_classical-LK-ODE}, ranging in a convex cone. In this way, the parametric method has been used a lot in the study of the class~$\clS$, see, e.g. \cite{Aleksandrov}, \cite[Sect.\,6]{BCMV_survey}, and references therein.
We also mention recent studies \cite{JuliaSeb,DVK,OliverSeb}, which make essential use of Theorem~\ref{TH_classical} and its analogue for hydrodynamically normalized univalent self-maps of the upper half-plane. Finally, it is worth mentioning that the univalence comes to de Branges' proof of Bieberbach's famous conjecture~\cite{deBranges} solely via a slight modification of Theorem~\ref{TH_classical}.

In contrast to linear operations, the operation of composition preserves univalence. With this fact being one of the cornerstones of Loewner's parametric method,  univalent self-maps of $\UD$ come to the scene. So let $\mathfrak S$ be a set of univalent maps $\phi:\UD\to\UD$ containing~$\id_\UD$ and closed w.r.t. composition, i.e. satisfying $\psi\circ\phi\in\mathfrak S$ whenever $\phi,\psi\in\mathfrak S$. The following {\it heuristic scheme} of parametric representation of~$\mathfrak S$ goes back to Loewner~[\mcite{Loewner}\,--\,\mcite{LoewnerMonotone}].

First we consider the set $\mathrm T\mathfrak S$ of all infinitesimal generators of one-parameter (sub)semi\-groups in $\mathfrak S$, see Section~\ref{SS_one-param} for precise definitions. Suppose that under a certain condition on a family $[0,+\infty)\ni t\mapsto G\in\mathrm T\mathfrak S$, the \hbox{time-$T$} map $\varphi_T^G$ of the non-autonomous ODE $dw/dt=G_t(w)$, $w(0)=z\in\UD$, belongs to~$\mathfrak S$  for any $T\ge0$. Consider the set $\mathfrak S^{L}$ of all \hbox{time-$T$} maps obtained in this way (when neither $T\ge0$ nor the family $(G_t)$ is fixed).
 
 To establish a {\it Loewner-type parametric representation} of~$\mathfrak S$ means to find an appropriate condition on families~$(G_t)$ that guarantees existence of a unique \hbox{time-$T$} map $\varphi_T^G\in\mathfrak S$ for all $T\ge0$, and to show that~$\mathfrak S^L=\mathfrak S$. This approach to parametric representation of univalent self-maps has been systematically developed in~[\mcite{Goryainov}\,--\,\mcite{Goryainov-Kudryavtseva}].

For the study of holomorphic self-maps $\phi:\UD\to\UD$, in particular from the view point of dynamics, an important role is played by the so-called boundary (regular) fixed points, see, e.g. \cite{MilnVas, AlexanderClarkMeasures, ContrMadrigal:AnaFlows, CDP1, CDP2, CowenPommerenke, Frolova, GoryainovFr, PommVas}, just to mention some studies on this topic. Combining Loewner's scheme of parametric representation with the intrinsic approach in Loewner Theory proposed in~\cite{BCM1,BCM2},  we will establish Loewner-type parametric representation for semigroups of univalent self-maps ${\phi:\UD\to\UD}$ defined by prescribing a finite set of points on~$\UC$ which are fixed by~$\phi$.

To conclude the Introduction, let us mention that Loewner's scheme makes perfect sense also for univalent self-maps of the unit ball and polydisk in~$\C^n$.  Loewner-type parametric representations for univalent functions in~$\UD$ make essential use of the Riemann Mapping Theorem, which is not available in several complex variables. By this reason, even for simplest choices of $\mathfrak S$, e.g.,  for the semigroup consisting of all univalent self-maps $\phi$ of the unit ball $\mathbb B^n:=\{z\in\C^n:\|z\|<1\}$, $n>1$, satisfying $\phi(0)=0$, to describe the reachable set $\mathfrak S^L$ is a hard open problem. It seems that the only known necessary condition for $\phi$ to be an element of~$\mathfrak S^L$ is that the domain $\phi(\mathbb B^n)$ has to be Runge in~$\C^N$, see~\cite[Section~4]{Erlend}.

\section{Main results}\label{S_results}
Denote by $\U$ the semigroup formed by all univalent holomorphic mappings ${\phi:\UD\to\UD}$. Further, for $\tau\in\overline\UD$, by $\Ut\tau$ we denote the subsemigroup of~$\U$ formed by $\id_\UD$ and all ${\phi\in\U\setminus\{\id_\UD\}}$ for which $\tau$  is the Denjoy\,--\,Wolff point, see Definition~\ref{DF_DW-point}.
Main result of this paper is the analogue of the classical Theorem~\ref{TH_classical} for two families of subsemigroups in~$\U$ and in~$\Ut\tau$ defined below.

Given a finite set $F\subset\UC$, by $\UF[F]$ we denote the subsemigroup of~$\U$ consisting all ${\phi\in\U}$ for which every element of $F$ is a boundary regular fixed point, see Definitions~\ref{DF_boundaryFP} and~\ref{DF_BRFP}. Let $\Ut\tau[F]:=\Ut\tau\cap\UF[F]$ for any $\tau\in\overline\UD\setminus F$.

Note that the semigroups defined above are always non-trivial, i.e. different from~$\{\id_\UD\}$, see, e.g. \cite[Example~3.1]{Gumenyuk_parametric}.

\begin{theorem}\label{TH_main}
For any finite set $F\subset \UC$ and any $\tau\in\overline\UD\setminus F$, the semigroups $\mathfrak S=\UF[F]$ and $\mathfrak S=\Ut\tau[F]$ admit Loewner-type parametric representation, i.e.
there exists a convex cone $\mathcal M_{\mathfrak S}$ of Herglotz vector fields in~$\UD$  with the following properties:
 \begin{mylist}
 \item[(i)] for every $G\in\mathcal M_{\mathfrak S}$ and a.e.~$s\ge0$,  $G(\cdot,s)$ is the infinitesimal generator of a one-parameter semigroup in~$\mathfrak{S}$;
\item[(ii)] for every $\phi\in\mathfrak S$ there exists $G\in\mathcal M_{\mathfrak S}$ such that $\phi(z)=w^G_{z,0}(1)$ for all~${z\in\UD}$, where $w=w^G_{z,s}(t)$ is the unique solution to the initial value problem
    \begin{equation*}%\label{EQ_LK-type-ini}
      \frac{dw(t)}{dt}=G\big(w(t),t\big),\quad w(s)=z,\qquad z\in\UD,\quad t\ge s\ge0;
    \end{equation*}
\item[(iii)]  for every $G\in\mathcal M_{\mathfrak S}$, the mappings $z\mapsto w_{z,s}^G(t)$, $t\ge s\ge0$, belong to~$\mathfrak S$.
\end{mylist}
\end{theorem}
Definitions and basic theory regarding Herglotz vector fields, infinitesimal generators, and one-parameter semigroups can be found in  Section~\ref{S_prelim}.

\begin{remark}\label{RM_about-the-convex-cone}
The convex cone $\mathcal M_{\mathfrak S}$ in Theorem~\ref{TH_main} can be characterized explicitly. For  ${\mathfrak S=\UF[F]}$, it coincides with the set of all Herglotz vector fields $G$ such that:\\[-2ex]
\begin{mylist}
\item[(a)]~$G$ satisfies condition~(i) in Theorem~\ref{TH_main}, which, due to~\cite[Theorem~1]{CDP2}, is equivalent to the existence of a finite angular limit
   $
      \lambda(\sigma,s):=\anglim_{z\to\sigma}G(z,s)/({z-\sigma})
    $
    for a.e.~${s\ge0}$ and every $\sigma\in F$;
\item[(b)] the functions $\lambda(\sigma,\cdot)$, $\sigma\in F$, are all locally integrable on~$[0,+\infty)$.
\end{mylist}

The case $\mathfrak S=\Ut\tau[F]$ is similar. The only difference is that in (a), in addition to having a finite angular limit $\lambda(\sigma,s)$ for a.e.~$s\ge0$ and every $\sigma\in F$, we should require that for a.e.~$s\ge0$, $G(\cdot,s)$~is the infinitesimal generator of a one-parameter semigroup~$(\phi_t^{G(\cdot,s)})$ for which $\tau$ is the Denjoy\,--\,Wolff point.
\end{remark}

For $\Un[\{1\}]$, Theorem~\ref{TH_main} was proved by Goryainov~\cite{GoryainovBRFP}. For~$\Ut\tau[F]$ with an arbitrary finite set $F\subset\UC$ and $\tau\in\UD$, it has been proved by the author in~\cite{Gumenyuk_parametric}. Moreover, the case of $\UF[F]$ with $\Card(F)\le3$ and that of $\Ut\tau[F]$ with $\Card(F)\le2$ and $\tau\in\UC$ are also covered in~\cite{Gumenyuk_parametric}. The proof in the remaining cases is presented in this paper. It is based on the following theorem.

\begin{theorem}\label{TH_embedd}
Let $F\in\UC$ be any finite set with $\Card(F)\ge3$ and let $\tau\in\UC\setminus F$. Then for any $\phi\in\Ut\tau[F]$ there exists an evolution family $(\varphi_{s,t})\subset\Ut\tau[F]$ such that $\phi=\varphi_{0,1}$.
\end{theorem}

The proofs of Theorems 1 and 2 are given in Section~\ref{S_proofs}. Sections~\ref{S_AC-measure} and~\ref{S_lemmas} contain some lemmas and other auxiliary statements, and in the next section we recall some basic definitions and facts from Holomorphic Dynamics and  Loewner Theory in the unit disk used throughout the paper.

\section{Preliminaries}\label{S_prelim}
\subsection{Contact and fixed points of holomorphic self-maps.}
Let $\phi\in\Hol(\UD,\UD)$. By the Schwarz Lemma, $\phi$ can have at most one fixed point in~$\UD$. However, in the sense of angular limits, there can be many (boundary) fixed points on $\UC$.
\begin{definition}\label{DF_boundaryFP}
A point $\sigma\in\UC$ is said to be a \textsl{boundary fixed point} of~$\phi$ if the angular limit $\anglim_{z\to\sigma}\phi(z)$ exists and coincides with~$\sigma$. More generally, $\sigma\in\UC$ is a \textsl{contact point} of~$\phi$ if it is a boundary fixed point of~$e^{i\theta}\phi$ for a suitable $\theta\in\Real$.
\end{definition}
 Following a usual convention,  if $\anglim_{z\to\xi}\phi(z)$ exists for some~$\xi$,  it will be denoted, in what follows, by $\phi(\xi)$ and  $\phi'(\xi)$ will stand for the angular derivative of~$\phi$ at~$\xi$, i.e. $\phi'(\xi):=\anglim_{z\to\xi}(\phi(z)-\phi(\xi))/(z-\xi)$, again provided that the latter limit exists.

 It is known that for any contact point $\sigma$ of a holomorphic self-map $\phi:\UD\to\UD$, the angular derivative $\phi'(\sigma)$ exists, finite or infinite, and if it is finite, then $\sigma\phi'(\sigma)/\phi(\sigma)>0$; in particular, $\phi'(\sigma)>0$ or $\phi'(\sigma)=\infty$ at any boundary fixed point~$\sigma$ of~$\phi$. See, e.g. \cite[Theorem 1.2.7 on p.\,53]{Abate} or \cite[Proposition~4.13 on p.\,82]{Pommerenke2}.

 \begin{definition}\label{DF_BRFP}
 A boundary fixed point (or a contact point) $\sigma$ of a holomorphic self-map ${\phi:\UD\to\UD}$ is said to be \textsl{regular} if the angular derivative $\phi'(\sigma)$ is finite.
 \end{definition}

Among all  fixed points of a self-map $\phi\in\Hol(\UD,\UD)\setminus\{\id_\UD\}$ there is one point of special importance for dynamics. Namely, by the classical Denjoy\,--\,Wolff Theorem, see, e.g. \cite[Theorem~1.2.14, Corollary~1.2.16,
Theorem~1.3.9]{Abate}, $\phi$ has a unique (boundary or interior) fixed point $\tau\in\overline\UD$ such that $|\phi'(\tau)|\le1$. Moreover, the sequence of iterates $(\phi^{\circ n})$, $\phi^{\circ 1}:=\phi$, $\phi^{\circ(n+1)}:=\phi\circ\phi^{\circ n}$, converges (to the constant function equal) to~$\tau$ locally uniformly in~$\UD$, unless $\phi$ is an elliptic automorphism.

\begin{definition}\label{DF_DW-point}
The point~$\tau$ above is referred to as the \textsl{Denjoy\,--\,Wolff point} of~$\phi$.
\end{definition}

\begin{remark}\label{RM_ChainRule}
A  version of the \textsl{Chain Rule for angular derivatives} states that if $\sigma$ is a contact point of~$\phi\in\Hol(\UD,\UD)$ and $\omega:=\phi(\sigma)$ is a contact point of $\psi\in\Hol(\UD,\UD)$, then $\sigma$ is a contact point of $\psi\circ\phi$ and $\big(\psi\circ\phi\big)'(\sigma)=\psi'(\omega)\cdot\phi'(\sigma)$; in particular, $\sigma$ is a regular contact point of $\psi\circ\phi$ provided $\sigma$ and $\omega$ are regular for~$\phi$ and $\psi$, respectively. See, e.g. \cite[Lemma~1.3.25 on p.\,92]{Abate}. It follows that the classes of univalent holomorphic self-maps defined in the Introduction, $\UF[F]$ and $\Ut\tau[F]$, are semigroups w.r.t. the operation~${(\psi,\phi)\mapsto\psi\circ\phi}$.
\end{remark}

\subsection{One-parameter semigroups  in $\Hol(\UD,\UD)$}\label{SS_one-param}
An important role in this study is played by the time-continuous analogue of interation, represented by one-parameter semigroups in $\Hol(\UD,\UD)$. For further details and the proofs of the results stated in this subsection we refer the reader to~\cite{ContrMadrigal:AnaFlows,CDP1} and eventually to references cited there; see also \cite[Chapter 3]{Shoikhet:2001}.
\begin{definition}
A family $(\phi_t)_{t\ge0}\subset\Hol(\UD,\UD)$ is said to be a \textsl{one-parameter semigroup} if $t\mapsto\phi_t$ is a continuous semigroup homomorphism from the $\big([0,+\infty),\cdot+\cdot\big)$ with the usual Euclidian topology to $\big(\Hol(\UD,\UD),\cdot\circ\cdot\big)$ endowed with the topology of locally uniform convergence in~$\UD$.

Equivalently, $(\phi_t)$ is a one-parameter semigroup if $\phi_0=\id_\UD$, $\phi_{t+s}=\phi_t\circ\phi_s=\phi_s\circ\phi_t$ for any
$t,s\ge0$, and if $\phi_t(z)\to z$ as $t\to0^+$ for any $z\in\UD$.
\end{definition}

All elements of a one-parameter semigroup $(\phi_t)$ different from~$\id_\UD$ have the same Denjoy\,--\,Wolff  point and the same set of boundary fixed points. Moreover, if a boundary fixed point $\sigma$ is regular for \textit{some} of such $\phi_t$'s, then $\sigma$ is regular for all~$\phi_t$'s.

\begin{remark}\label{RM_Berkson-Porta}
It is known~\cite{Berkson-Porta} that for any one-parameter semigroup $(\phi_t)$ the limit
$$
G(z):=\lim_{t\to0^+}\frac{\phi_t(z)-z}{t}
$$
exists for all~$z\in\UD$ and it is a holomorphic function in~$\UD$, referred to as the \textsl{infinitesimal generator} of~$(\phi_t)$. Moreover, $G$ admits the following representation
\begin{equation}\label{EQ_Berkson-Porta}
G(z)=(\tau-z)(1-\overline\tau z)p(z),\quad z\in\UD,
\end{equation}
where $\tau\in\overline\UD$ and $p\in\Hol(\UD,\C)$ with $\Re p\ge0$ are determined by~$(\phi_t)$ uniquely unless $\phi_t=\id_\UD$ for all~$t\ge0$: namely, $\tau$ is the Denjoy\,--\,Wolff point of each $\phi_t$ different from~$\id_\UD$.
Furthermore, for any $z\in\UD$, the function $w=w_z(t):=\phi_t(z)$, $t\ge0$, is the unique solution to the initial value problem
\begin{equation}\label{EQ_BP-IVP}
\frac{dw(t)}{dt}=G(w(t)),\quad t\ge0;\qquad w(0)=z.
\end{equation}
Conversely, if $G\in\Hol(\UD,\C)$ for any $z\in\UD$ the unique solution to \eqref{EQ_BP-IVP} extends to all~${t\ge0}$ (i.e., in other words, $G$ is a holomorphic semicomplete autonomous vector field in $\UD$), then $G$ is of the form~\eqref{EQ_Berkson-Porta} and hence corresponds via~\eqref{EQ_BP-IVP} to a uniquely defined one-parameter semigroup, which we will call \textsl{generated by}~$G$ and will denote by~$(\phi_t^G)$.
\end{remark}
Representation~\eqref{EQ_Berkson-Porta} is known as the \textsl{Berkson\,--\,Porta formula}. 

\subsection{Basics of modern Loewner Theory}
An elementary, but important consequence of Remark~\ref{RM_Berkson-Porta} is that all elements of any one-parameter semigroup in $\Hol(\UD,\UD)$ are~{\it univalent}. However, it is known, see, e.g. \cite{GESR}, that not every univalent $\phi:\UD\to\UD$ is an element of a one-parameter semigroup.  According to Loewner's idea, in order to embed a given univalent self-map~$\phi$ into a semiflow, one should consider a \textit{non-autonomous} version of~\eqref{EQ_BP-IVP}, i.e.
\begin{equation}\label{EQ_LK-ODE}
\frac{dw(t)}{dt}=G\big(w(t),t\big),\quad t\ge s\ge 0;\qquad w(s)=z\in\UD.
\end{equation}
But for which class of functions $G:\UD\times[0,+\infty)$ are the corresponding (non-autonomous) semiflows of~\eqref{EQ_LK-ODE} defined globally, i.e. for any $z\in\UD$ and any $t\ge s\ge0$?
Attempts to answer this question have led to the following definition~\cite[Section~4]{BCM1}.

\begin{definition}\label{DF_HVF}
A function $G:\mathbb{D}\times[0,+\infty)\to \mathbb{C}$ is called a \textsl{Herglotz vector field (in the unit disk)} if it satisfies the following conditions:

\begin{mylist}
\item[HVF1.] For every $z\in\mathbb{D}$, the function $[0,+\infty)\ni t\mapsto G(z,t)$ is measurable.

\item[HVF2.] For a.e. $t\in[0,+\infty),$ the function $\mathbb{D}\ni z\mapsto G_t(z):=G(z,t)$ is an infinitesimal generator, i.e. $G_t$ admits the Berkson\,--\,Porta representation~\eqref{EQ_Berkson-Porta}.

\item[HVF3.] For any compact set $K\subset\mathbb{D}$ and any $T>0$ there
exists a non-negative locally integrable function $k_{K,T}$ on~$[0,+\infty)$ such that $|G(z,t)|\le k_{K,T}(t)$ for all $z\in K$ and a.e.~$t\in[0,T].$
\end{mylist}
\end{definition}

 In what follows, we will assume that $G$ in~\eqref{EQ_LK-ODE} is a Herglotz vector field. In such a case  by \cite[Theorem\,4.4]{BCM1}, for any $z\in\UD$ and any~$s\ge0$, the initial value problem~\eqref{EQ_LK-ODE} has a unique solution $w=w^G_{z,s}(t)$ defined for all~$t\ge s$; and the differential equation in that problem is called the \textsl{(generalized) Loewner\,--\,Kufarev ODE}. Equation~\eqref{EQ_classical-LK-ODE} in Theorem~\ref{TH_classical} and equation~\eqref{EQ_BP-IVP} related to one-parameter semigroups are its special cases. (The former equation is often referred as the \textsl{(classical radial) Loewner\,--\,Kufarev ODE} or simply as the \textsl{radial Loewner ODE.})

 Similarly to one-parameter semigroups, the class of all families $(\varphi_{s,t})_{t\ge s\ge0}$ which are (non-autonomous) semiflows of Herglotz vector fields~$G$, in the sense that $\varphi_{s,t}(z)=w^G_{z,s}(t)$ for all $z\in\UD$, all $s\ge0$, and all $t\ge s$, can be characterized intrinsically. Namely, according to \cite[Theorem 1.1]{BCM1}, initial value problem~\eqref{EQ_LK-ODE} defines a one-to-one correspondence between Herglotz vector fields~$G$ and evolution families~$(\varphi_{s,t})$, with the latter concept defined as follows.

\begin{definition}[\protect{\cite[Definition~3.1]{BCM1}}]\label{DF_evol_family}
A family $(\varphi_{s,t})_{t\ge s\ge0}\subset\Hol(\UD,\UD)$ is called an {\sl evolution
family (in the unit disk)} if it satisfies the following conditions:
\begin{mylist}
\item[EF1.] $\varphi_{s,s}=\id_{\mathbb{D}}$ for any~$s\ge0$.

\item[EF2.] $\varphi_{s,t}=\varphi_{u,t}\circ\varphi_{s,u}$ whenever $0\le
s\le u\le t$.

\item[EF3.] For all $z\in\mathbb{D}$ and  $T>0$ there exists an
integrable function $k_{z,T}:[0,T]\to[0,+\infty)$
such that
\[
|\varphi_{s,u}(z)-\varphi_{s,t}(z)|\le\int_{u}^{t}k_{z,T}(\xi)\DI\xi
\]
whenever $0\le s\le u\le t\le T.$
\end{mylist}
\end{definition}

Now we give a precise definition of what we mean by a Loewner-type parametric representation.
\begin{definition}[\protect{\cite[Definition~2.6]{Gumenyuk_parametric}}]\label{DF_Loewner-type_representation}
We say that a subsemigroup~$\mathfrak{S}\subset\Hol(\UD,\UD)$ \textsl{admits Loewner-type parametric representation} if there exists a convex cone $\mathcal M_{\mathfrak S}$ of Herglotz vector fields in~$\UD$ with the following properties:
\begin{mylist}
\item[LPR1.] For every $G\in\mathcal M_{\mathfrak S}$, we have that $G_t:=G(\cdot,t)$ for a.e.~$t\ge0$ is the generator of a one-parameter semigroup in~$\U$.

\item[LPR2.] The evolution family~$(\varphi^G_{s,t})$ of any~$G\in\mathcal M_{\mathfrak S}$  satisfies $(\varphi^G_{s,t})\subset\mathfrak S$.

\item[LPR3.] For every $\phi\in\mathfrak S$ there exists $G\in\mathcal M_{\mathfrak S}$ such that $\phi=\varphi^G_{s,t}$ for some $s\ge0$ and~$t\ge s$, where $(\varphi^G_{s,t})$ stands, as above, for the evolution family of the Herglotz vector field~$G$.
\end{mylist}
\end{definition}
Although the above definition can be stated without using concept of evolution family, we will take essential advantage of the correspondence between Herglotz vector fields and evolution families. Namely, in the proof of Theorem~\ref{TH_main}, to check  condition LR3  we will use Theorem~\ref{TH_embedd} and results from~\cite{Gumenyuk_parametric} to construct a suitable evolution family containing~$\phi$ and then prove that its Herglotz vector field belongs to~$\mathcal M_{\mathfrak S}$.

The convex cone $\mathcal M_{\mathfrak S}$ defined in Remark~\ref{RM_about-the-convex-cone} has the property that if $G\in\mathcal M_{\mathfrak S}$, then $(z,t)\mapsto G(z,at+b)$ also belongs to~$\mathcal M_{\mathfrak S}$ for any~$a>0$ and any~$b\ge0$. This fact allows us to replace, in Theorem~\ref{TH_main}, the condition~$\phi=\varphi_{s,t}$ from LPR3 by condition~$\phi=\varphi_{0,1}$.

The interplay between evolution families and Herglotz vector fields is completed by constructing related one-parameter families of conformal maps~$f_t:\UD\to\C$, $t\ge0$, known as Loewner chains.
\begin{definition}[\protect{\cite[Definition 1.2]{SMP}}]\label{DF_L-ch}
A family $(f_t)_{t\ge0}\subset\Hol(\UD,\C)$  is said to be a {\sl Loewner chain (in the unit disk)} if the following three conditions hold:
\begin{enumerate}
\item[LC1.] Each function $f_t:\D\to\C$ is univalent.

\item[LC2.] $f_s(\D)\subset f_t(\D)$ whenever $0\le s < t$.

\item[LC3.] For any compact set $K\subset\mathbb{D}$,  there exists a
non-negative locally integrable function $k_{K}:[0,+\infty)\to\Real$ such that
\[
|f_s(z)-f_t(z)|\leq\int_{s}^{t}k_{K}(\xi)d\xi
\]
for all $z\in K$ and all $s,t\ge0$ with $t\ge s$.
\end{enumerate}
\end{definition}

We will use the fact, see \cite[Theorem~1.3]{SMP}, that given a Loewner chain $(f_t)$,  the functions~$\varphi_{s,t}:=f_t^{-1}\circ f_s$, $t\ge s\ge 0$, form an evolution family, which is said to be \textsl{associated with} the Loewner chain~$(f_t)$.

\section{Regular contact points and Alexandrov\,--\,Clark measure}\label{S_AC-measure}
 The following theorem is a slight reformulation of the characterization for existence of finite angular derivative given in~\cite[p.\,52--53]{Sarason}.
\begin{result}\label{TH_Sarason}
Let $\psi\in\Hol(\UD,\UD)$ and let $\mu$ be the Alexanderov\,--\,Clark measure of~$\psi$ at~$1$, i.e.
$$
\frac{1+\psi(z)}{1-\psi(z)}=\int_\UC\frac{\sigma+z}{\sigma-z}\,\di\mu(\sigma)\,+\,iC,\quad\text{for all~$z\in\UD$},\qquad C:=\Im\frac{1+\psi(0)}{1-\psi(0)}.
$$
Then for any $\sigma_0\in\UC$, the following two conditions are equivalent:
\begin{itemize}
\item[(i)]  $\sigma_0$ is a regular contact point of~$\psi$ and $\psi(\sigma_0)\neq 1$;
\item[(ii)] the function $\UC\ni\sigma\mapsto|\sigma-\sigma_0|^{-2}$ is $\mu$-integrable.
\end{itemize}
\end{result}

\begin{corollary}\label{CR_Sarason}
Let $\Psi\in\Hol(\UH,\UH)$ be given by
\begin{equation}\label{EQ_Nevanlinna0}
\Psi(\zeta)=\alpha+\beta \zeta+\int_{\Real}\frac{1+t\zeta}{t-\zeta}\,\DI\nu(t)\quad\text{for all~$\zeta\in\UH$,}
\end{equation}
where $\alpha\in\Real$ and $\beta\ge0$ are some constants and $\nu$ is a finite Borel measure on~$\Real$.
If~$x_0\in\Real$ is a regular contact point with $\Psi(x_0)\neq\infty$, then $\Real\ni t\mapsto |1+tx_0|/|t-x_0|$ is $\nu$-integrable and
\begin{equation*}
\Psi(x_0)=\alpha+\beta x_0+\int_{\Real}\frac{1+t x_0}{t-x_0}\,\DI\nu(t).
\end{equation*}
\end{corollary}
\begin{proof}
Write $\psi:=H^{-1}\circ\Psi\circ H$, where $H(z):=i(1+z)/(1-z)$ is the Cayley transform of~$\UD$ onto~$\UH$. Using the relation between the measure~$\nu$ and the Alexandrov\,--\,Clark measure~$\mu$ of~$\psi$ at~$1$, see, e.g. \cite[p.\,138]{Bhatia}, we immediately deduce from Theorem~\ref{TH_Sarason} that $t\mapsto|t-x_0|^{-2}$ is $\nu$-integrable. Since $\nu$ is finite, it follows that the functions $t\mapsto|1+tx_0|/|t-x_0|$, $t\mapsto |t|/|t-x_0|^2$ and $t\mapsto (1+t^2)/|t-x_0|^2$ are also $\nu$-integrable. It remains to separate the real and imaginary parts in~\eqref{EQ_Nevanlinna0} for $\zeta:=x_0+i\varepsilon$, $\varepsilon>0$, and use Lebesgue's Dominated Convergence Theorem to pass to the limit as $\varepsilon\to0^+$.
\end{proof}

\section{Lemmas}
\label{S_lemmas}
\REM{
\begin{lemma}\label{LM_rational-function}
Let $t_1\le t_2<t_3\le t_4$ be real numbers, with $(t_2-t_1)^2+(t_4-t_3)^2\neq0$. Then for any $t\in(t_2,t_3)$ and any~$t'\in\Real\setminus[t_2,t_3]$,
$$
Q(t):=\frac{(t-t_1)(t-t_4)}{(t-t_2)(t-t_3)}~>~Q(t').
$$
\end{lemma}
\begin{proof}
If $t_2=t_1$ or $t_4=t_3$, then $Q$ is a linear-fractional function and the desired conclusion follows from the fact that $Q$ is monotonic on each of the intervals $(-\infty,p)$ and $(p,+\infty)$, where $p$ stands for the unique pole of~$Q$.

Suppose now that $t_2\neq t_1$ and $t_4\neq t_3$.
Then $Q(t)\to+\infty$ as $(t_2,t_3)\ni t\to t_2$ and as $(t_2,t_3)\ni t\to t_3$. Hence the range of $Q$ on $(x_2,x_3)$ is of the form $[a_0,+\infty)$ and the equation $Q(t)=a$ has at least two roots in $(x_2,x_3)$ for any $a>a_0$ and at least one multiple root in $(x_2,x_3)$ if ${a=a_0}$. Taking into account that $Q$ is a rational function of degree two,  it follows that $Q\big(\ComplexE\setminus(t_2,t_3)\big)\cap[a_0,+\infty)=\emptyset$. Therefore, for any $t'\in\Real\setminus[t_2,t_3]$, we have $Q(t')<a_0$. The proof is complete.
\end{proof}}

Let us fix a finite set $F\subset\UC\setminus\{1\}$ with $n:=\Card(F)\ge3$. Write $F$ as a finite sequence $(\sigma_j)_{j=1}^n$ satisfying $0<\arg \sigma_1<\arg\sigma_2<\ldots\arg \sigma_n<2\pi$. Denote by $L_j$, $j=1,\ldots,n-1$, the open arc of~$\UC$ going from~$\sigma_j$ to $\sigma_{j+1}$ in the counter-clockwise direction and let $L_0$ stand for the arc between $\sigma_1$ and $\sigma_n$ containing the point~$1$.

Let $\phi\in\Uone[F]\setminus\{\id_\UD\}$. Denote by $\mathcal C(L)$  the set of all $\psi\in\Hol(\UD,\UD)$ that extends continuously to an open arc~$L\subset\UC$ with $\psi(L)\subset\UC$. By~\cite[Lemma~5.6(ii)]{Gumenyuk_parametric}, there is no $j=1,\ldots,n-1$ such that that $\phi\in\mathcal C(L_j)$. It may, however, happen that $\phi\in\mathcal C(L_0)$.

\begin{remark}\label{RM_contact-preim}
Suppose now that $f$ is any conformal mapping of~$\UD$ with $\phi(\UD)\subset f(\UD)\subset \UD$. Then for each~$j=1,\ldots,n$ there exists a unique regular contact point~$\xi_j$ of~$f$ such that $f(\xi_j)=\sigma_j$, see, e.g. \cite[Theorem~4.14 on p.\,83]{Pommerenke2}. Throughout the paper, we will use this statement implicitly and write $f^{-1}(\sigma_j)$ instead of~$\xi_j$, most of the times without a reference to this remark. Similarly, by $f^{-1}(1)$ we denote the unique regular contact point at which the angular limit of~$f$ equals~$1$.
\end{remark}

\begin{remark}\label{RM_contact-from-LCh}
It follows immediately from the above remark and~\cite[Lema~5.1]{Gumenyuk_parametric} that if $\phi(\UD)\subset f_1(\UD)\subset f_2(\UD)\subset \UD$ for some univalent maps $f_1,f_2:\UD\to\C$, then $\psi:=f_2^{-1}\circ f_1$ has a regular contact point at $f^{-1}_1(\sigma)$ for any $\sigma\in F\cup\{1\}$ and $\psi\big(f^{-1}_1(\sigma)\big)=f^{-1}_2(\sigma)$.
\end{remark}

First, applying \cite[Lemma~5.5]{Gumenyuk_parametric}, we are going to construct a multi-parameter ``inclusion chain''~$(g_a)_{a\in[0,1]^n}$, which later will be used to obtain an evolution family in~$\Uone[F]$ containing~$\phi$. In what follows, $a$ and $b$ denote elements of~$[0,1]^n$ with components $a_0,a_1,\ldots,a_{n-1}$ and $b_0,b_1,\ldots,b_{n-1}$, respectively. By $|L|$ we will denote the length of a circular arc or a line segment~$L$.

\begin{lemma}\label{LM_multiparam_family}
In the above notation, there exists a family~$(g_a)_{a\in[0,1]^n}$ of  univalent holomorphic self-maps of~$\UD$ satisfying the following conditions:
\begin{mylist}
\item[(g1)] for every $a\in[0,1]^n$, $g_a\in\UF[\{1,\sigma_1,\sigma_n\}]$;

\item[(g2)] $g_{(0,\ldots,0)}=\phi$ and $g_{(1,\ldots,1)}=\id_\UD$;

\item[(g3)] if $a,b\in[0,1]^n$, $a_j\le b_j$ for all $j=0,\ldots,{n-1}$, then $g_a(\UD)\subset g_b(\UD)$, with the strict inclusion if $a_k<b_k$ for some $k\in[1,n-1]\cap\N$;

\item[(g4)] if $k\in[0,n-1]\cap\Natural$ and if $a,b\in[0,1]^n$ with $a_k\le b_k$ and $a_j=b_j$ for all $j=0,\ldots,n-1$, $j\neq k$, then $\psi_{a,b}:=g_b^{-1}\circ g_a\in\mathcal C\big(\UC\setminus C_k(a)\big)$, with $\psi_{a,b}(\UC\setminus C_k(a))=\UC\setminus C_k(b)$, where for $k>0$  and $x\in[0,1]^n$, $C_k(x)$  denotes the closed arc of~$\UC$ going counter-clockwise from~$g_x^{-1}(\sigma_k)$ to~$g^{-1}_x(\sigma_{k+1})$ and $C_0(x):=L_0$ for any~$x\in[0,1]^n$;

\item[(g5)] $g_a$ does not depend on~$a_0$ if and only if $\phi\in\mathcal C(L_0)$; and if these two equivalent assertions fail, then the inequality ${a_0<b_0}$ becomes a sufficient condition for the strict inclusion in~(g3).

\item[(g6)] for any map $t\in[0,+\infty)\mapsto a(t)=\big(a_0(t),\ldots,a_{n-1}(t)\big)\in [0,1]^n$ with absolutely continuous and non-decreasing components $t\mapsto a_j(t)$, $j=0,\ldots,n-1$, the functions $\varphi_{s,t}:=g_{a(t)}^{-1}\circ g_{a(s)}$, $t\ge s\ge0$, form an evolution family;

\item[(g7)]  the maps $[0,1]^n\ni a\mapsto g_a'(1)$ and $[0,1]^n\ni a\mapsto g_a^{-1}(\sigma_j) $, $j=2,\ldots,{n-1}$, are separately continuous in each of the variables $a_0$, $a_1$,\ldots, $a_n$.
\end{mylist}
\end{lemma}
\begin{proof}
Let us first suppose that $\phi$ fails to extend continuously to~$L_0$ with $\phi(L_0)\subset\UC$. Then by  \cite[Lemma~5.5]{Gumenyuk_parametric}, applied with $z_0:=0$ and $L_n:=L_0$, there exists a family $(f_a)_{a\in[0,1]^n}\subset\U$ such that:
\begin{itemize}
\item[(A)] assertions (g3), (g4), and (g5) hold with $(g_a)$ replaced by~$(f_a)$;
\item[(B)] $f_{(0,\ldots,0)}=\phi$ and $f_{(1,\ldots,1)}\in\Aut(\UD)$;
\item[(C)] $f_a'(0)=\phi(0)$ and $f'_a(0)\overline{\phi'(0)}>0$ for all $a\in[0,1]^n$;
\item[(D)] the map $[0,1]^n\ni a\mapsto R\big(f_a(\UD),\phi(0)\big)$, where $R(D,w_0)$ stands for the conformal radius of~$D$ w.r.t.~$w_0$, is Lipschitz continuous.
\end{itemize}
For each~$a\in[0,1]^n$, consider the unique $h_a\in\Aut(\UD)$  that takes the points $\sigma_0:=1$, $\sigma_1$, and $\sigma_n$ to the points $f^{-1}_a(1)$, $f^{-1}_a(\sigma_1)$, and $f^{-1}_a(\sigma_n)$, respectively. Then the family $(g_a)$ defined by $g_a:=f_a\circ h_a$ satisfy conditions~(g1)\,--\,(g5).

To prove~(g6), we notice that from (A), (C), and (D) it follows that $(F_t)_{t\ge0}$ defined by $F_t:=f_{a(t)}$ for all~$t\in[0,+\infty)$ is a Loewner chain, see, e.g.
\cite[proof of Theorem~2.3]{CAOT}. Let $(\Phi_{s,t})$ be the evolution family associated with~$(F_t)$.  Then $\eta_j:=f_{a(0)}^{-1}(\sigma_j)$, $j=0,\ldots,n$, are regular contact points of~$(\Phi_{s,t})$ in the sense of~\cite[Definition~3.1]{BRFPLoewTheory}.
By~\cite[Theorem~3.5]{BRFPLoewTheory}, the maps $[0,+\infty)\ni t\mapsto \Phi_{0,t}(\eta_j)=f_t^{-1}(\sigma_j)$, ${j=0,\ldots,n}$, are locally absolutely continuous.  Therefore, with the help of \cite[Lemma~2.8]{SMP} we conclude that the functions $$\varphi_{s,t}:=g_{a(t)}^{-1}\circ g_{a(s)}=h_{a(t)}^{-1}\circ\Phi_{s,t}\circ h_{a(s)},\quad t\ge s\ge0,$$ form an evolution family.

It remains to prove~(g7).
 To this end fix some~$k\in[0,n-1]\cap\N$ and apply the above argument to the function $t\mapsto a(t)=\big(a_0(t),\ldots,a_{n-1}(t)\big)$ with the components $a_k=\min\{t,1\}$ and $a_j=a_j^0$ for all $j=0,\ldots,n-1$, $j\neq k$, where $a_j^0$'s are some arbitrary numbers in~$[0,1]$. As above, by~\cite[Theorem~3.5]{BRFPLoewTheory}, it follows that $t\mapsto |f'_{a(t)}(1)|=|f'_{a(0)}(1)|/|\Phi_{0,t}'(\eta_0)|$  and $t\mapsto f_{a(t)}^{-1}(\sigma_j)=\Phi_{0,t}(\eta_j)$, $j=0,\ldots,n$, are continuous. This fact immediately implies~(g7).

The proof for the case in which $\phi$ extends continuously to~$L_0$ with $\phi(L_0)\subset\UC$ is similar except that the family~$(g_a)$ we construct depends only on~$n-1$ parameters $a_j$, each corresponding to one of the arcs~$L_1,\ldots, L_{n-1}$. In this case, for each $a\in[0,1]^n$, $g_a\in\mathcal{C}(L_0)$ with $g_a(L_0)=L_0$. Adding the parameter~$a_0$ formally gives the family~$(g_a)$ satisfying all the conditions~(g1)--(g7).
\end{proof}

Now we are going to reduce the number of the parameters by choosing $a_0$ to be a suitable function of~$a':=(a_1,\ldots,a_{n-1})$. As usual, we will identify $(a_0,a')$ with $(a_0,a_1,\ldots,a_{n-1})$.

\begin{lemma}\label{LM_monotonic}
Let, as above, $\phi\in\Uone[F]\setminus\{\id_\UD\}$ and let $(g_a)_{a\in[0,1]^n}$ be a family of  univalent holomorphic self-maps of~$\UD$ satisfying conditions (g1)\,--\,(g5) and (g7) in Lemma~\ref{LM_multiparam_family}. Then there exists a map $[0,1]^{n-1}\ni a':=(a_1,\ldots,a_{n-1})\mapsto a_0(a')\in[0,1]$ with the following properties:
\begin{itemize}
\item[(n1)] $a_0(0,\ldots,0)=0$ and $a_0(1,\ldots,1)=1$;
\item[(n2)] $a'\mapsto a_0(a')$ is continuous and non-decreasing in each $a_j$, $j=1,\ldots,n-1$;
\item[(n3)] $a'\mapsto \lambda\big(a_0(a'),a'\big)$, where $\lambda(a):=g'_{a}(1)$, is non-decreasing in each $a_j$, ${j=1,\ldots,n-1}$.
\end{itemize}
\end{lemma}
\begin{proof}
Fix for a moment $a,b\in[0,1]^n$ with $a_j\le b_j$ for all~$j=0,\ldots,n-1$. Then $\psi_{a,b}:=g_b^{-1}\circ g_a\in\UF[1,\sigma_1,\sigma_n]$ thanks to conditions~(g1)\,--\,(g3) and Remark~\ref{RM_contact-from-LCh}. Moreover, by the Chain Rule for angular derivatives, see Remark~\ref{RM_ChainRule}, $\lambda(a)=\lambda(b)\psi'_{a,b}(1)$.

Suppose first that $\phi\in\mathcal C(L_0)$. Then $g_a$ does not depend on $a_0$ and $\psi_{a,b}\in\mathcal C(L_0)$. By~\cite[Lemma~5.6(i)]{Gumenyuk_parametric}, $\psi'_{a,b}(1)\le1$ and hence $\lambda(a)\le \lambda(b)$. This show that $\lambda$ is non-decreasing in variables~$a_1$,\ldots, $a_{n-1}$ and (trivially) does not depend on~$a_0$. Therefore, in case $\phi\in\mathcal C(L_0)$, we may choose any map $a'\mapsto a_0(a')$ satisfying (n1) and (n2), e.g.,  $$\textstyle a_0(a'):=\left(\sum_{j=1}^{n-1}a_j\right)/(n-1).$$

So from now on we may suppose that $\phi\not\in\mathcal C(L_0)$. Let $a_0=b_0$. Then $\psi_{a,b}\in\mathcal C(L_0)$ by~(g4) and arguing as above, we see that $\psi_{a,b}'(1)\le 1$. Moreover, if additionally $a\neq b$, then by~(g3), $\psi_{a,b}\neq\id_\UD$ and hence \cite[Lemma~5.6(i)]{Gumenyuk_parametric} yields the strict inequality $\psi_{a,b}'(1)<1$. Therefore, $\lambda$ is increasing in variables~$a_1$,\ldots, $a_{n-1}$. Now let $a'=b':=(b_1,\ldots,b_{n-1})$ and $a_0<b_0$. Arguing in a similar way, but using the second assertion of \cite[Lemma~5.6]{Gumenyuk_parametric}, we see that $\psi_{a,b}'(1)>1$, and hence $\lambda$ is strictly decreasing in~$a_0$.
Moreover, recall that by~(g7), $\lambda$ is continuous separately in each variable $a_0$, $a_1$,\ldots,$a_{n-1}$. Finally, note that separate continuity and monotonicity in each variable imply joint continuity of $a\mapsto \lambda(a)$, see, e.g. \cite[Remark~5.7]{Gumenyuk_parametric}.

Fix some $a'\in[0,1]^{n-1}$. Note that $\lambda(0,a')\le\lambda(1,a')\le\lambda(1,1,\ldots,1)=1$. Therefore, because of monotonicity and continuity in~$a_0$, there exists a unique $a_0(a')\in[0,1]$ such that
\begin{equation}\label{EQ_eq-for-a0}
\lambda(a_0(a'),a')=\min\big\{\lambda(0,a'),\,1\big\}.
\end{equation}
As the minimum of two continuous non-decreasing functions, the right-hand side of~\eqref{EQ_eq-for-a0} is continuous and non-decreasing in~$a_1$,\ldots,$a_{n-1}$, which proves~(m3) and continuity of ${a'\mapsto a_0(a')}$. To see that this map is monotonic, note that by construction $a_0(a')=0$ if ${\lambda(0,a')\le1}$ and that $a_0=a_0(a')\in[0,1]$ solves the equation $\lambda(a_0,a')=1$ if $\lambda(0,a')>1$. Fix $a',b'\in[0,1]^{n-1}$ with $a_j\le b_j$ for all $j=1,\ldots,n-1$. Taking into account that it is not possible to have $\lambda(0,a')>1$ and $\lambda(0,b')\le 1$ at the same time, careful analysis of the remaining three cases shows that $a_0(a')\le a_0(b')$.

By (g2), $\lambda(0,0,\ldots,0)=\phi'(1)\le 1$ and hence $a_0(0,\ldots,0)=0$. Similarly, again by~(g2), $\lambda(1,1,\ldots,1)=1$. Therefore, $\lambda(0,1,\ldots,1)>1$ and we can easily conclude that $a_0(1,\ldots,1)=1$. This completes the proof.
\end{proof}

The following lemma can be viewed as one of the possible analogues of Loewner's Lemma for mappings with boundary Denjoy\,--\,Wolff point.
Fix some points $\xi_j\in\UC$, $j=1,\ldots,n$, ordered counter-clockwise in such a way that the arc of~$\UC$ between $\xi_1$ and~$\xi_n$ that contains $\tau=1$ does not contain the points $\xi_2$,\ldots,$\xi_{n-1}$. Denote by $L'_j$, $j=1,\ldots,n-1$, the open arc of~$\UC$ going counter-clockwise from~$\xi_j$ to~$\xi_{j+1}$.
\begin{lemma}\label{LM_Loewner-type}
In the above notation, let $\psi\in\Uone[\{\xi_1,\xi_n\}]\setminus\{\id_\UD\}$ and suppose that $\xi_2,\ldots,\xi_{n-1}$ are regular contact points of~$\psi$. Let $\Psi:=H\circ\psi\circ H^{-1}$ and $x_j:=H(\xi_j)$, $j=1,\ldots,n$, where $H(z):=i(1+z)/(1-z)$ is the Cayley transform of~$\UD$ onto the upper half-plane~$\UH$. Let $k\in[1,n-1]\cap\N$. If $\psi\in\mathcal C(L'_j)$ for all~$j=1,\ldots,n-1$, $j\neq k$, then
\begin{align}
 \label{EQ_monot1}
 \Psi(x_{k+1})-\Psi(x_k) &< x_{k+1}-x_k,\quad\text{and}\\
 \label{EQ_monot2}
 \Psi(x_{j+1})-\Psi(x_j) &> x_{j+1}-x_j,\quad\text{for all~$j=1,\ldots,n-1$, $j\neq k$}.
\end{align}
\end{lemma}
\begin{proof}
The Nevalinna representation of $\Psi$ is,
see, e.g. \cite[p.\,135--142]{Bhatia}:
\begin{equation}\label{EQ_Nevanlinna}
\Psi(z)=\alpha+\beta z+\int_{\Real}\frac{1+t\zeta}{t-\zeta}\,\DI\nu(t)\quad\text{for all~$\zeta\in\UH$,}
\end{equation}
where $\alpha\in\Real$, $\beta:=1/\phi'(1)\ge 1$ and $\nu$ is a finite positive Borel measure on~$\Real$.
In accordance with Corollary~\ref{CR_Sarason}, for $j=1,\ldots,n-1$ we have
\begin{equation}\label{EQ_xm-xl}
\Psi(x_{j+1})-\Psi(x_j)=(x_{j+1}-x_{j})\Big[\beta+\int\limits_\Real\frac{1+t^2}{(t-x_{j})(t-x_{j+1})} \,\DI\nu(t)\Big].
\end{equation}
If $j\neq k$, then by hypothesis, $\psi\in\mathcal C(L'_j)$, and hence $\nu([x_{j},x_{j+1}])=0$. With this taken into account, \eqref{EQ_xm-xl} implies \eqref{EQ_monot2} with the sign $>$ replaced by~$\ge$. To check that actually the strict inequality holds, we note that the equality is possible only if $\beta=1$ and $\nu(\Real)=0$. In such a case, taking into account that $\Psi(x_1)=x_1$, from~\eqref{EQ_Nevanlinna} we would get $\Psi=\id_\UH$, which contradicts the hypothesis.
Finally~\eqref{EQ_monot1} follows immediately because
$$
\sum_{j=1}^{n-1}\big(\Psi(x_{j+1})-\Psi(x_j)\big)=\Psi(x_n)-\Psi(x_1)=x_n-x_1=\sum_{j=1}^{n-1}(x_{j+1}-x_j).
$$\vskip-3ex
\end{proof}

 Applying the above lemma, we can now pass in the multi-parameter family~$(g_a)_{a\in[0,1]}$ to a unique parameter in such a way that the points~$\sigma_j$, $j=1,\ldots,n$, are kept fixed and the angular derivative at~$\tau=1$ is non-increasing. For $a\in[0,1]^n$ and $j=1,\ldots,n$, denote $\xi_j(a):=g^{-1}_a(\sigma_j)$.

\begin{lemma}\label{LM_pass-to-one-parameter}
Under hypothesis of Lemma~\ref{LM_monotonic}  there exists a continuous map $[0,1]\ni \theta\mapsto a(\theta)=\big(a_0(\theta),a_1(\theta),\ldots,a_n(\theta)\big)\in[0,1]^n$ such that:
\begin{itemize}
\item[(m1)] $\theta\mapsto a_j(\theta)$ is strictly increasing for each~$j=1,\ldots n$ and non-decreasing for $j=0$;
\item[(m2)] $a(0)=(0,\ldots,0)$ and $a(1)=(1,\ldots,1)$;
\item[(m3)] $\xi_j\big(a(\theta)\big)=\sigma_j$ for each $j=1\ldots,n$ and any~$\theta\in[0,1]$;
\item[(m4)] $\theta\mapsto \lambda(a(\theta))=g_{a(\theta)}'(1)$ is non-decreasing.
\end{itemize}
\end{lemma}
\begin{proof}
Consider the family $(h_{a'})_{a'\in[0,1]^{n-1}}$, defined by $h_{a'}:=g_{(a_0(a'),a')}$, where $a'\mapsto a_0(a')$ is the map defined in Lemma~\ref{LM_monotonic}.

Given $a',b'\in[0,1]^{n-1}$ with $a_j\le b_j$ for $j=1,\ldots,n$, consider the map ${\psi_{a',b'}:=h^{-1}_{b'}\circ h_{a'}}$. The points ${\tau=1}$ and $\xi_j(a'):=h_{a'}^{-1}(\sigma_j)$, $j=1,\ldots,n$, are regular contact points of~$\psi_{a',b'}$, with $\psi_{a',b'}(1)=1$ and $\psi_{a',b'}(\xi_j(a'))=\xi_j(b')$, see Remark~\ref{RM_contact-from-LCh}.  In particular, $$\xi_1(a')=\xi_1(b')=\sigma_1,\qquad \xi_n(a')=\xi_n(b')=\sigma_n$$ and hence $\psi\in\UF[1,\sigma_1,\sigma_2]$. Taking into account  that  $$h'_{a'}(1)=\lambda(a_0(a'),a')\le \lambda(a_0(b'),b')=h'_{b'}(1)$$ by Lemma~\ref{LM_monotonic}, we conclude that  $\psi_{a',b'}\in\Uone[\sigma_1,\sigma_n]$. Moreover, thanks to condition~(g3), $\psi_{a',b'}\neq\id_\UD$ unless $a'=b'$.

Applying Lemma~\ref{LM_Loewner-type} to the map $\psi:=\psi_{a',b'}$ with regular contact points $\xi_j:=\xi_j(a')$,  we see that for any ${k=1,\ldots,n-1}$, the function $$a'=(a_1,\ldots,a_{n-1})\mapsto \ell_k(a'):=H(\xi_{k+1}(a'))-H(\xi_k(a'))$$ is strictly decreasing in $a_k$ and strictly increasing in each $a_j$ with $j\neq k$. Moreover, by~(g7), this function is separately continuous in each variable, which in view of monotonicity, implies the joint continuity, see, e.g.~\cite[Remark~5.7]{Gumenyuk_parametric}.

 To prove the lemma, it is sufficient to construct a continuous map $$[0,1]\ni \theta\mapsto a'(\theta)=\big(a_1(\theta),\ldots,a_{n-1}(\theta)\big)\in[0,1]^{n-1}$$ satisfying the following conditions:
 \begin{itemize}
 \item[(p1)] all the components $\theta\mapsto a_j(\theta)$, ${j=1,\ldots,n-1}$, are strictly increasing;
 \item[(p2)] $a'(0)=(0,\ldots,0)$, $a'(1)=(1,\ldots,1)$, and
 \item[(p3)] the maps $\theta\mapsto\ell_j(a'(\theta))$, $j=1,\ldots,n-1$, are constant.
 \end{itemize}

The construction is based on an recursive procedure. First we will find a continuous function $[0,1]^{n-2}\ni a'':=(a_1,\ldots,a_{n-2})\mapsto a_{n-1}(a'')\in[0,1]$, $a_{n-1}(0,\ldots,0)=0$, $a_{n-1}(1,\ldots,1)=1$, strictly increasing in each variable and such that $\ell_{n-1}(a'',a_{n-1}(a''))$ is constant. Namely, for any $a''\in[0,1]^{n-2}$ we set $a_{n-1}(a'')$ to be a solution to
$$
\ell_{n-1}(a'',a_{n-1})=H(\sigma_n)-H(\sigma_{n-1}),
$$
which exists and unique thanks to continuity and monotonicity of $\ell_{n-1}$ in $a_{n-1}$ and to the fact that by monotonicity in~$a_1$,\ldots,$a_{n-2}$ we have
 \begin{align}
  \label{EQ_align1}
 \ell_{n-1}(a'',0)\ge\ell_{n-1}(0,\ldots,0)&={H(\sigma_n)-H(\sigma_{n-1})}\quad \text{and}\\
   \label{EQ_align2}
 \ell_{n-1}(a'',1)\le\ell_{n-1}(1,\ldots,1)&={H(\sigma_n)-H(\sigma_{n-1})}.
 \end{align}
 The required properties of continuity and monotonicity of $a''\mapsto a_{n-1}(a'')$ follow from continuity and monotonicity of~$\ell_{n-1}$. Finally, using again \eqref{EQ_align1} and~\eqref{EQ_align2} one immediately  obtains equalities ${a_{n-1}(0,\ldots,0)=0}$ and ${a_{n-1}(1,\ldots,1)=1}$.

For $n=3$ the proof can be now completed by taking $a(\theta):=\big(a_0(\theta,a_2(\theta)),\theta,a_2(\theta)\big)$ for all~$\theta\in[0,1]$. So suppose that ${n>3}$. It is easy to see that for any $k=1,\ldots,n-2$, the function  $\ell^1_k(a''):=\ell_k(a'',a_{n-1}(a''))$, defined for all $a''=(a_1,\ldots,a_{n-2})\in [0,1]^{n-2}$, is continuous and strictly increasing in  each $a_j$ with $j\neq k$. Since
$$
\sum_{j=1}^{n-2}\ell_j^1(a'')=H(\sigma_n)-H(\sigma_1)-\ell_{n-1}\big(a''\!,a_{n-1}(a'')\big)=H(\sigma_{n-1})-H(\sigma_1),
$$
it follows immediately that $\ell^1_k$ is strictly decreasing in~$a_k$. Furthermore,
\begin{align*}
\ell_{n-2}^1(0,\ldots,0)&=\ell_{n-2}(0,\ldots,0,0)={H(\sigma_{n-1})-H(\sigma_{n-2})}\quad\text{and}\\ \ell_{n-2}^1(1,\ldots,1)&=\ell_{n-2}(1,\ldots,1,1)={H(\sigma_{n-1})-H(\sigma_{n-2})}.
\end{align*}
Therefore, the above argument can be applied again with $\ell_{n-1}$ replaced by $\ell^1_{n-2}$. In other words,  we exclude one more parameter by finding $a_{n-2}=a_{n-2}(a''')\in[0,1]$, $a''':=(a_1,\ldots,a_{n-3})$, that solves the equation $\ell^1_{n-2}\big(a''',a_{n-2}\big)={H(\sigma_{n-1})-H(\sigma_{n-2})}.$

Repeating this procedure suitable number of times, we end up with a continuous map $$[0,1]\ni\theta\mapsto a'(\theta)=\big(\theta,a_2(\theta),\ldots,a_{n-1}(\theta)\big)\in[0,1]^{n-1}$$  that satisfies (p1)\,--\,(p3).
Using conditions (n1)\,--\,(n3) in Lemma~\ref{LM_monotonic}, we see that the map~$\theta\mapsto a(\theta):=\big(a_0(a'(\theta)),a'(\theta)\big)$ satisfies~(m1)\,--\,(m4). The proof is now complete.
\end{proof}

\section{Proof of main results}\label{S_proofs}
\subsection{Proof of Theorem~\ref{TH_embedd}}
Without loss of generality we may assume that $\tau=1$. Furthermore, for $\phi=\id_\UD$ the statement of the theorem is trivial, so will will suppose that~$\phi\neq\id_\UD$.

Now we can apply Lemma~\ref{LM_multiparam_family} to construct the multi-parameter family $(g_a)_{a\in[0,1]^n}$. Next using Lemma~\ref{LM_pass-to-one-parameter}, we obtain the continuous map $[0,1]\ni\theta\mapsto a=\big(a_0(\theta),\ldots,a_{n-1}(\theta)\big)$. By (m1) and~(m2), the function $$\Lambda(\theta):=\frac{1}{n}\sum_{j=0}^{n-1}a_j(\theta),\qquad \theta\in[0,1],$$ is a continuous strictly increasing map of~$[0,1]$ onto itself. Let $[0,1]\ni t\mapsto\theta(t)$ be the inverse of the function~$\Lambda$, which we extend to $[0,+\infty)$ by setting~$\theta(t):=1$ for all~$t>1$. Clearly, the functions $t\mapsto a_j(\theta(t))$, $j=0,\ldots,n-1$, are non-decreasing and $$|a_j(\theta(t_2))-a_j(\theta(t_1))|\le\sum_{k=0}^{n-1}|a_k(\theta(t_2))-a_k(\theta(t_1))|\le|t_2-t_1|$$ for any~$t_1,t_2\ge0$ and any $j=0,\ldots,n-1$. Therefore, by assertion~(g6) of Lemma~\ref{LM_multiparam_family}, the functions $$\varphi_{s,t}:=g_{a(\theta(t))}^{-1}\circ g^{\phantom{-1}}_{a(\theta(t))},\qquad t\ge s\ge0,$$ form an evolution family. Moreover, since $a(\theta(0))=(0,\ldots,0)$ and $a(\theta(1))=(1,\ldots,1)$, by assertion~(g2) of Lemma~\ref{LM_multiparam_family} we have $\varphi_{0,1}=\phi$.

 Bearing in mind Remarks~\ref{RM_ChainRule} and~\ref{RM_contact-from-LCh}, it remains to mention that $(\varphi_{s,t})\subset \UF[F]$ by~(m3) and that $(\varphi_{s,t})\subset\Uone$ thanks to~(m4).
\proofbox

\subsection{Proof of Theorem~\ref{TH_main}}
It is sufficient to prove the theorem for $\mathfrak S=\UF[F]$ or for ${\mathfrak S=\Ut\tau[F]}$, where $F\subset\UC$ is a finite set with $\Card(F)\ge3$ and $\tau\in\UC\setminus F$, as all other cases are already covered by~\cite[Theorem~1]{Gumenyuk_parametric}.

Fix any Herglotz vector field~$G$. By the very definition, for a.e.~$s\ge0$, $G(\cdot,s)$ is an infinitesimal generator. By~\cite[Theorem~1]{CDP2}, the one-parameter semigroup $(\phi_{t}^{G(\cdot,s)})$ generated by~$G(\cdot,s)$ has a boundary regular fixed point at~$\sigma\in\UC$ if and only if there exists a finite angular limit
   \begin{equation}\label{EQ_anglim-for-G}
      \lambda(\sigma,s):=\anglim_{z\to\sigma}\frac{G(z,s)}{z-\sigma}.
    \end{equation}

Now let $\mathcal M_{\mathfrak S}$ be the set of all Herglotz vector fields~$G$ satisfying the following conditions:
\begin{mylist}
\item[(a)] for a.e. $s\ge0$, we have $(\phi_{t}^{G(\cdot,s)})\subset\mathfrak S$;
\item[(b)] for every $\sigma\in F$, the function $\lambda(\sigma,\cdot)$ is locally integrable on~$[0,+\infty)$.
\end{mylist}

First of all, let us show that $\mathcal M_{\mathfrak S}$ is a convex cone. Let $G_1,G_2\in\mathcal M_{\mathfrak S}\setminus\{0\}$ and let~$G_3\not\equiv0$ be a linear combination of $G_1$ and $G_2$ with non-negative coefficients. Then $G_3$ is an infinitesimal generator, see, e.g. \cite[Corollary~1.4.15 on p.\,108]{Abate}. Moreover, if for some $s\ge0$, $\tau$ is the Denjoy\,--\,Wolff point of both $(\phi_{t}^{G_1(\cdot,s)})$ and $(\phi_{t}^{G_2(\cdot,s)})$, then thanks to the Berkson\,--\,Porta formula, see Remark~\ref{RM_Berkson-Porta}, $\tau$ is also the Denjoy\,--\,Wolff point of $(\phi_{t}^{G_3(\cdot,s)})$.
Finally, since  the angular limit~\eqref{EQ_anglim-for-G} exists finitely and satisfies~(b) both for $G:=G_1$ and for $G:=G_2$, this is the case for~$G:=G_3$ as well. Therefore, $G_3\in\mathcal M_{\mathfrak S}$.

It remains to show that $\mathcal M_{\mathfrak S}$ meets conditions~(i)\,--\,(iii) in Theorem~\ref{TH_main}. Firstly, condition~(i) simply coincides with~(a). Furthermore, in view of \cite[Theorem~1]{CDP2} and \cite[Theorem~1.1]{BRFPLoewTheory}, conditions~(a) and (b) imply that the evolution family~$(\varphi_{s,t}^G)$ generated by~$G$ satisfies $(\varphi_{s,t}^G)\subset\UF[F]$. This proves (iii) in case $\mathfrak S=\UF[F]$. Notice that if $\mathfrak S=\Ut\tau[F]$, then in view of the Berkson\,--\,Porta formula, see Remark~\ref{RM_Berkson-Porta}, from~(a) it follows that $G(z,t)=(\tau-z)(1-\overline\tau z)p_t(z)$ for all~$z\in\UD$ and a.e.~$t\ge0$, where $(p_t)_{t\ge0}$ is a family of holomorphic functions with non-negative real part. By \cite[Corollary~7.2]{BCM1}, the latter equality implies that $(\varphi_{s,t}^G)\subset\Ut\tau$. Hence (iii) also holds for the case $\mathfrak S=\Ut\tau[F]$.

It remains show that $\mathcal M_{\mathfrak S}$ satisfies~(ii). Let $\phi\in\mathfrak S\setminus\{\id_\UD\}$. We have to find  $G\in\mathcal M_{\mathfrak S}$ whose evolution family~$(\varphi_{s,t})=(\varphi^G_{s,t})$ satisfies $\varphi_{0,1}=\phi$. To this end we first construct  a certain evolution family~$(\varphi_{s,t})$ and then show that its Herglotz vector field is a suitable candidate for~$G$. Let $\xi$ be the Denjoy\,--\,Wolff point of~$\phi$. Of course, $\xi=\tau$ if $\mathfrak S=\Ut\tau[F]$.

\vskip.7ex
\noindent{\bf Claim:} \textit{there exists an evolution family $(\varphi_{s,t})\subset\UF[F]\cap\Ut\xi$ such that $\varphi_{0,1}=\phi$.}
\vskip.7ex

\noindent Indeed, if $\xi\in\UC\setminus F$, then the Claim follows readily from Theorem~\ref{TH_embedd} applied with $\tau$ replaced with~$\xi$. Similarly, if $\xi\in F$ and $\Card(F)>3$, then we should apply Theorem~\ref{TH_embedd} with $\tau$ and~$F$ replaced with $\xi$ and $F\setminus\{\xi\}$, respectively. In the remaining cases, i.e. if $\xi\in F$ and $\Card(F)=3$ or if $\xi\in\UD$,  the proof of the Claim is contained in the proof of \cite[Theorem~1]{Gumenyuk_parametric}, again with $\xi$ and $F\setminus\{\xi\}$ substituted for $\tau$ and $F$, respectively.

\vskip.7ex
Now by the above Claim, there exists an evolution family $(\varphi_{s,t})\subset\mathfrak S$ with $\varphi_{0,1}=\phi$. By \cite[Theorem~1]{CDP2} and \cite[Theorem~1.1]{BRFPLoewTheory}, the Herglotz vector field~$G$ of~$(\varphi_{s,t})$ belongs to~$\mathcal M_{\UF[F]}$. This completes the proof for $\mathfrak S=\UF[F]$, while in case $\mathfrak S=\Ut\tau[F]$ it remains to notice that $(\phi_t^{G(\cdot,s)})\subset \Ut\tau$ for a.e.~$s\ge0$ by~\cite[Theorem~6.7]{BCM1}.
\proofbox


\begin{thebibliography}{99}
%
\bibitem{Abate}M.\,Abate, {\it Iteration theory of holomorphic maps on taut manifolds},
Research and Lecture Notes in Mathematics. Complex Analysis and Geometry,
Mediterranean, Rende, 1989.

\bibitem{Aleksandrov} I.\,A.\,Aleksandrov,
\textit{Parametric continuations in the theory of univalent functions}
(Russian), Izdat. ``Nauka'', Moscow, 1976. \hbox{MR0480952}

\bibitem{Erlend} L. Arosio, F. Bracci\ and\ E.\,F.~Wold, {\it Solving the Loewner PDE in complete hyperbolic starlike domains of $\mathbb{C}^N$}, Adv. Math. {\bf 242} (2013). MR3055993

\bibitem{MilnVas}J. M. Anderson\ and\ A. Vasil'ev, {\it Lower Schwarz-Pick estimates and angular derivatives}, Ann. Acad. Sci. Fenn. Math. {\bf 33} (2008), no.~1, 101--110. MR2386840 (2009b:30014)

\bibitem{Berkson-Porta}E. Berkson and H. Porta, \textit{Semigroups of
holomorphic functions and composition operators,} Michigan Math. J. \textbf{25} (1978), 101--115.

\bibitem{Bhatia} R. Bhatia, {\it Matrix analysis}, Graduate Texts in Mathematics, 169, Springer, New York, 1997. MR1477662

\bibitem{AlexanderClarkMeasures}F. Bracci, M. D. Contreras\ and\ S. D\'\i az-Madrigal, {\it Aleksandrov-Clark measures and semigroups of analytic functions in the unit disc}, Ann. Acad. Sci. Fenn. Math. {\bf 33} (2008), no.~1, 231--240. MR2386848 (2009c:30097)

\bibitem{BCM1} F. Bracci, M. Contreras, S. D\'iaz-Madrigal, {\it Evolution Families and the Loewner Equation I: the unit disc}. J. Reine Angew. Math. (Crelle's Journal), {\bf 672} (2012), 1--37.

\bibitem{BCM2} \bysame, {\it Evolution Families and the Loewner Equation II: complex
hyperbolic manifolds}. Math. Ann. {\bf 344} (2009), No.\,4, 947--962.

\bibitem{BRFPLoewTheory} F.~Bracci, M.\,D. Contreras, S.~D\'\i az-Madrigal, and P.~Gumenyuk,
 {\it Boundary regular fixed points in Loewner theory}, Ann. Mat. Pura Appl. (4) {\bf 194} (2015), no.~1, 221--245.

\bibitem{BCMV_survey} F.~Bracci, M.\,D. Contreras, S.~D\'\i az-Madrigal, and A.~Vasil'ev, {\it Classical and stochastic L\"owner-Kufarev equations}, in Harmonic and complex analysis and its applications, 39--134, Trends Math, Birkh\"auser/Springer, Cham, 2014. MR3203100



\bibitem{deBranges} L. de Branges, {\it A proof of the Bieberbach conjecture},
Acta Math. {\bf 154} (1985), no.~1-2, 137--152. MR0772434 (86h:30026)



\bibitem{ContrMadrigal:AnaFlows} M.\,D. Contreras and S.~D\'{\i}az-Madrigal,
\emph{Analytic flows in the unit disk: angular derivatives and
boundary fixed points}, Pacific J. Math. \textbf{222} (2005), 253--286.

\bibitem{SMP} M.\,D. Contreras, S.~D\'\i az-Madrigal, and P.~Gumenyuk, \textit{Loewner chains in the unit disk}. Rev. Mat. Iberoam. \textbf{26} (2010), 975--1012.

\bibitem{CAOT}
\bysame, \emph{Geometry behind chordal Loewner chains}, Complex Anal. Oper. Theory {\bf 4} (2010), no.~3, 541--587. MR2719792

\bibitem{CDP1}M. D. Contreras, S. D\'\i az-Madrigal\ and\ C. Pommerenke, {\it Fixed points and boundary behaviour of the Koenigs function}, Ann. Acad. Sci. Fenn. Math. {\bf 29} (2004), no.~2, 471--488. MR2097244

\bibitem{CDP2} \bysame, {\it On boundary critical points for semigroups of analytic functions}, Math. Scand. {\bf 98} (2006), no.~1, 125--142. MR2221548


\bibitem{CowenPommerenke} C. C. Cowen\ and\ C. Pommerenke, {\it Inequalities for the angular derivative of an analytic function in the unit disk}, J. London Math. Soc. (2) {\bf 26} (1982), no.~2, 271--289. MR0675170



\bibitem{GESR} M. Elin, V. Goryainov,  S.~Reich, D.~Shoikhet, {\it Fractional iteration and functional equations for functions analytic in the unit disk}, Comput. Methods Funct. Theory {\bf 2} (2002)no.~2, [On table of contents: 2004]. MR2038126


\bibitem{Frolova} A. Frolova, M. Levenshtein, D. Shoikhet, A. Vasil'ev, {\it Boundary distortion estimates for holomorphic maps}, Complex Anal. Oper. Theory {\bf 8} (2014), no.~5, 1129--1149. MR3208806



\bibitem{GoryainovFr}V. V. Gorya\u\i nov, {\it Fractional iterates of functions that are analytic in the unit disk with given fixed points} Mat. Sb. {\bf 182} (1991), no. 9, 1281--1299; translation in Math. USSR-Sb. {\bf 74} (1993), no.~1, 29--46. MR1133569 (92m:30049)

\bibitem{Goryainov}V.V. Goryainov, \textit{Semigroups of conformal mappings,}
Mat. Sb. (N.S.) {\bf 129(171)} (1986), no.~4, 451--472 (Russian); translation
in Math. USSR Sbornik \textbf{57} (1987), 463--483.

\bibitem{Goryainov1996}
\bysame, \textit{Evolution families of analytic functions and time-inhomogeneous
Markov branching processes}, Dokl. Akad. Nauk \textbf{347}(1996), No.\,6,
729--731; translation in Dokl. Math. \textbf{53}(1996), No.\,2, 256--258.

\bibitem{GoryainovDiff} \bysame, \textit{Evolution families of conformal mappings with fixed points.} (Russian. English summary). Z\'\i{}rnik Prats' Instytutu Matematyky NAN Ukrayiny. National Academy of Sciences of Ukraine (ISSN 1815-2910) {\bf 10} (2013), No.4-5, 424-431. Zbl~1289.30024

\bibitem{GoryainovObzor}\bysame, \textit{Semigroups of analytic functions in analysis and applications}, Uspekhi Mat. Nauk {\bf 67} (2012), no. 6(408), 5--52; translation in Russian Math. Surveys {\bf 67} (2012), no.~6, 975--1021

\bibitem{GoryainovBRFP} \bysame, \textit{Evolution families of conformal mappings with fixed points and the L\"owner-Kufarev equation}, Mat. Sb. {\bf 206} (2015), no. 1, 39--68; translation in Sb. Math. {\bf 206} (2015), no.~1-2, 33-60.

\bibitem{Goryainov-Ba}V.\,V. Goryainov and I. Ba, \textit{Semigroups of
conformal mappings of the upper half-plane into itself with hydrodynamic
normalization at infinity,} Ukrainian Math. J. \textbf{44} (1992), 1209--1217.

\bibitem{Goryainov-Kudryavtseva} V.\,V.~Goryainov, O.\,S.~Kudryavtseva, {\it One-parameter semigroups of analytic functions, fixed points and the Koenigs function}, Mat. Sb. \textbf{202} (2011), No.\,7, 43--74 (Russian); translation in Sbornik: Mathematics,
\textbf{202} (2011), No.\,7-8, 971--1000.


\bibitem{Gumenyuk_parametric}P.~Gumenyuk, {\it Parametric representation of univalent functions with boundary regular fixed points}. Preprint, 2016. Available at ArXiv:1603.04043.


\bibitem{JuliaSeb} J. Koch\ and\ S. Schlei\ss{}inger, {\it Value ranges of univalent self-mappings of the unit disc}, J. Math. Anal. Appl. {\bf 433} (2016) no.~2. MR3398791


\bibitem{Loewner}K. L\"{o}wner, {\it Untersuchungen \"{u}ber schlichte
konforme Abbildungen des Einheitskreises}, Math. Ann. \textbf{89} (1923),
103--121.

\bibitem{LoewnerMatrices} C. Loewner, \textit{On totally positive matrices,} Math. Z. {\bf 63} (1955), 338--340. MR0073657

\bibitem{LoewnerSeminar} \bysame, {\it Seminars on analytic functions}, Institute for Advanced Study, Princeton, New Jersey, vol.\,1, (1957). Available at \verb+http://babel.hathitrust.org+

\bibitem{LoewnerMonotone} \bysame, {\it On generation of monotonic transformations of higher order by infinitesimal transformations}, J. Analyse Math. {\bf 11} (1963), 189--206. MR0214711

\bibitem{Poggi} P. Poggi-Corradini, {\it Canonical conjugations at fixed points other than the Denjoy-Wolff point}, Ann. Acad. Sci. Fenn. Math. {\bf 25} (2000), no.~2, 487--499. MR1762433 (2001f:30033)

\bibitem{Pommerenke}Ch.\,Pommerenke, {\it Univalent functions}.
With a chapter on quadratic differentials by Gerd Jensen, Vandenhoeck \&
Ruprecht, G\"{o}ttingen, 1975.

\bibitem{Pommerenke2}\bysame, {\it Boundary behaviour of conformal mappings}. Springer-Verlag, 1992.

\bibitem{DVK}D. Prokhorov\ and\ K. Samsonova, {\it Value range of solutions to the chordal Loewner equation}, J. Math. Anal. Appl. {\bf 428} (2015) no.~2. MR3334955

\bibitem{PommVas} C. Pommerenke\ and\ A. Vasil'ev, {\it Angular derivatives of bounded univalent functions and extremal partitions of the unit disk}, Pacific J. Math. {\bf 206} (2002), no.~2, 425--450. MR1926785 (2003i:30024)



\bibitem{OliverSeb} O. Roth\ and\ S. Schlei\ss{}inger, {\it Rogosinski's lemma for univalent functions, hyperbolic Archimedean spirals and the Loewner equation}, Bull. Lond. Math. Soc. {\bf 46} (2014) no.~5. MR3262210


\bibitem{Sarason} D. Sarason, {\it Sub-Hardy Hilbert spaces in the unit disk}.
University of Arkansas Lecture Notes in the Mathematical Sciences, 10. A Wiley-Interscience Publication. John Wiley \& Sons, Inc., New York, 1994. MR1289670 (96k:46039)

\bibitem{Shoikhet:2001}
D.~Shoikhet, \emph{Semigroups in geometrical function theory}, Kluwer Academic Publishers, 2001.
\end{thebibliography}
\end{document}